\documentclass[a4paper,reqno]{amsart}

\usepackage{amsmath}
\usepackage{amssymb}
\usepackage{amsthm}
\usepackage{ascmac}
\usepackage{cases}
\usepackage{multicol}
\usepackage[all]{xy}
\usepackage{nccmath}
\usepackage{graphicx}
\usepackage{amscd}
\usepackage[all]{xy}
\usepackage{float}
\usepackage{bm} 
\usepackage{color}
\usepackage{enumerate}
\usepackage{lineno}
\usepackage{ulem}

\newtheorem{theorem}{Theorem}[section]
\newtheorem{proposition}[theorem]{Proposition}

\newtheorem{lemma}[theorem]{Lemma}
\theoremstyle{definition}
\newtheorem{definition}[theorem]{Definition}

\newtheorem{remark}[theorem]{Remark}

\newtheorem*{remark*}{Remark}

\newcommand{\type}{\operatorname{type}}

\newcommand{\tri}{\triangleleft}
\newcommand{\pr}{\operatorname{pr}}
\newcommand{\z}{\mathbb{Z}}

\begin{document}

\title[Affine extensions of MCQs and augmented MCQ Alexander pairs]
{Affine extensions of multiple conjugation quandles and augmented MCQ Alexander pairs}

\author[T.~Murao]{Tomo Murao}
\address[T.~Murao]{Global Education Center, Waseda University, 1-6-1 Nishi-Waseda, Shinjuku, Tokyo 169-8050, Japan}
\email{tmurao@aoni.waseda.jp}


\begin{abstract}
A multiple conjugation quandle is an algebra 
whose axioms are motivated from handlebody-knot theory.
Any linear extension of a multiple conjugation quandle 
can be described by using a pair of maps called an MCQ Alexander pair.
In this paper, 
we show that 
any affine extension of a multiple conjugation quandle 
can be described by using a quadruple of maps, called an augmented MCQ Alexander pair.
\end{abstract}

\maketitle

\section{Introduction}

A quandle~\cite{Joyce82,Matveev82} is an algebraic structure 
whose axioms are derived from the Reidemeister moves for oriented knots.
Ishii and Oshiro~\cite{IshiiOshiro19} introduced 
a pair (resp. triple) of maps called an Alexander pair (resp. augmented Alexander pair), 
which is a dynamical cocycle~\cite{AndruskiewitschGrana03} 
corresponding to a linear (resp. affine) extension of a quandle. 
A linear/affine extension of a quandle plays an important role 
in constructing knot invariants.
Actually, 
(twisted) Alexander invariants~\cite{Alexander28,Lin01,Wada94} 
and quandle cocycle invariants~\cite{CarterJelsovskyKamadaLangfordSaito03} 
for knots 
correspond to 
linear and affine extensions 
of quandles, respectively~\cite{IshiiOshiro19}.

A multiple conjugation quandle (MCQ)~\cite{Ishii15} is an algebraic structure 
whose axioms are derived from the Reidemeister moves for handlebody-knots~\cite{Ishii08}, 
which is a handlebody embedded in the 3-sphere $S^3$.
A genus 1 handlebody-knot can be regarded as a knot 
by taking that spine, 
which means that a handlebody-knot is a generalization of a knot with respect to a genus.
Recently, 
the author~\cite{Murao21} introduced an MCQ Alexander pair, 
which is an MCQ version of an Alexander pair.
That is, 
an MCQ Alexander pair is a pair of maps 
yielding a linear extension of an MCQ.
In this paper, 
we introduce a quadruple of maps, 
called an augmented MCQ Alexander pair, 
which is an MCQ version of an augmented Alexander pair.
That is, 
an augmented MCQ Alexander pair yields an affine extension of an MCQ.

Similarly to quandles, 
a linear/affine extension of an MCQ plays an important role 
in constructing handlebody-knot invariants. 
For example, 
(twisted) Alexander invariants~\cite{IshiiNikkuniOshiro18} for handlebody-knots 
are obtained through the theory of linear extensions of MCQs, 
and MCQ cocycle invariants~\cite{CarterIshiiSaitoTanaka17} for handlebody-knots 
are related to their affine extensions (see Section~2).

An MCQ is a quandle consisting of a disjoint union of some groups.
Hence an MCQ equips some binary operations: 
a quandle operation and each group operation.
For that reason, 
we need to deal with a 6-tuple of maps to consider an affine extension of an MCQ, 
and it is very complicated.
In this paper, 
we show that 
a 6-tuple of maps which gives an affine extension of an MCQ can be reduced
to some augmented MCQ Alexander pair modulo isomorphism.
That is, 
any affine extension of an MCQ 
can be realized using some augmented MCQ Alexander pair 
up to isomorphism.

The outline of the paper is as follows.
In Section~\ref{sec:Multiple conjugation quandles and augmented MCQ Alexander pairs}, 
we recall the notion of an MCQ Alexander pair 
and introduce an augmented MCQ Alexander pair.
We see that it is related to an extension of an MCQ.
Furthermore, 
we see that 
an MCQ 2-cocycle corresponds to 
 a special case of an augmented MCQ Alexander pair.
In Section~\ref{sec:Affine extensions of multiple conjugation quandles}, 
we consider affine extensions of MCQs.
We give a 6-tuple of maps 
corresponding to an affine extension of an MCQ.
In Section~\ref{sec:The reduction of affine extensions of MCQs to augmented MCQ Alexander pairs},
we show that 
any affine extension of an MCQ 
can be realized by using an augmented MCQ Alexander pair 
modulo isomorphism.

\section{Multiple conjugation quandles and augmented MCQ Alexander pairs}
\label{sec:Multiple conjugation quandles and augmented MCQ Alexander pairs}

A \textit{quandle}~\cite{Joyce82,Matveev82} is a pair of a set $Q$ 
and a binary operation $\triangleleft:Q\times Q\to Q$ satisfying the following axioms:
\begin{itemize}
\item[(Q1)]
For any $a\in Q$, $a\triangleleft a=a$.
\item[(Q2)]
For any $a\in Q$, the map $S_a:Q\to Q$ defined by $S_a(x)=x\triangleleft a$ is bijective.
\item[(Q3)]
For any $a,b,c\in Q$, $(a\triangleleft b)\triangleleft c=(a\triangleleft c)\triangleleft(b\triangleleft c)$.
\end{itemize}
We denote $S_a^n(x)$ by $x \triangleleft^n a$ 
for any $a,x \in Q$ and $n\in\mathbb{Z}$.
In the following, 
we see some examples of quandles.
Let $G$ be a group.
The \textit{conjugation quandle} of $G$, 
denoted by $\operatorname{Conj}G$, 
is defined to be the group $G$ 
with the binary operation $a\triangleleft b=b^{-1}ab$.
For a positive integer $n$, 
we denote by $\mathbb{Z}_n$ the cyclic group $\mathbb{Z}/n\mathbb{Z}$ of order $n$.
The \textit{dihedral quandle} of order $n$, denoted by $R_n$, 
is defined to be the cyclic group $\mathbb{Z}_n$ 
with the binary operation $a\triangleleft b=2b-a$.
Let $R$ be a ring.
The \textit{Alexander quandle} is defined to be a left $R[t^{\pm 1}]$-module 
with the binary operation $a\triangleleft b=ta+(1-t)b$.

For quandles $(Q_1,\triangleleft_1)$ and $(Q_2,\triangleleft_2)$, 
a \textit{quandle homomorphism} $f:Q_1\to Q_2$ is defined 
to be a map $f:Q_1\to Q_2$ satisfying $f(a\triangleleft_1 b)=f(a)\triangleleft_2 f(b)$ for any $a,b\in Q_1$.
We call a bijective quandle homomorphism a \textit{quandle isomorphism}.
$Q_1$ and $Q_2$ are \textit{isomorphic}, denoted by $Q_1 \cong Q_2$, 
if there exists a quandle isomorphism from $Q_1$ to $Q_2$.

We define the \textit{type} of a quandle $Q$, denoted by $\type Q$, 
by the minimal number of $n \in \z_{>0}$ satisfying $a\triangleleft^n b=a$ for any $a,b \in Q$.
We set $\type Q:= \infty$ if we do not have such a positive integer $n$.
Any finite quandle is of finite type 
since 
the set $\{S_a^n \mid n\in \z \}$ is finite for any $a \in Q$ when $Q$ is finite.
For a quandle $Q$, 
an \textit{extension} of $Q$ is a quandle $\widetilde{Q}$ 
which has a surjective homomorphism $f : \widetilde{Q} \to Q$ 
such that for any element of $Q$, 
the cardinality of the inverse image by $f$ is constant.

\begin{definition}[\cite{Ishii15}]
A \textit{multiple conjugation quandle (MCQ)} $X$ 
is a disjoint union of groups $G_{\lambda} (\lambda \in \Lambda)$ 
with a binary operation $\tri : X \times X \to X$ 
satisfying the following axioms:
\begin{itemize}
\item
For any $a,b \in G_\lambda$, 
$a\tri b = b^{-1}ab$.
\item
For any $x \in X$ and $a,b \in G_\lambda$, 
$x \tri e_\lambda = x$ and $x \tri (ab)=(x \tri a) \tri b$, 
where $e_\lambda$ is the identity of $G_\lambda$.
\item
For any $x,y,z \in X$, 
$(x \tri y) \tri z=(x \tri z) \tri (y \tri z)$.
\item
For any $x \in X$ and $a,b \in G_\lambda$, 
$(ab) \tri x=(a \tri x)(b \tri x)$, 
where $a \tri x, b \tri x \in G_\mu$ for some $\mu \in \Lambda$.
\end{itemize}
\end{definition}

In this paper, we often omit brackets.
When doing so, we apply binary operations from left on expressions, 
except for group operations, 
which we always apply first.
For example, 
we write $a \triangleleft_1 b \triangleleft_2 cd \triangleleft_3 (e \triangleleft_4 f \triangleleft_5 g)$ 
for $((a \triangleleft_1 b) \triangleleft_2 (cd)) \triangleleft_3 ((e \triangleleft_4 f) \triangleleft_5 g)$ simply, 
where each $\triangleleft_i$ is a binary operation, 
and $c$ and $d$ are elements of the same group.
Furthermore, 
unless otherwise specified, 
we assume that each $G_\lambda$ is a group 
when $\bigsqcup_{\lambda \in \Lambda}G_\lambda$ is an MCQ.
We denote by $G_a$ the group $G_\lambda$ containing $a \in X$.
We also denote by $e_\lambda$ the identity of $G_\lambda$.
Then the identity of $G_a$ is denoted by $e_a$ for any $a \in X$.

We remark that an MCQ itself is a quandle.
For MCQs $X_1=\bigsqcup_{\lambda \in \Lambda}G_\lambda$ and $X_2=\bigsqcup_{\mu \in M}G_\mu$, 
an \textit{MCQ homomorphism} $f : X_1 \to X_2$ is defined to be a map 
from $X_1$ to $X_2$ satisfying $f(x \tri y)=f(x) \tri f(y)$ for any $x,y \in X_1$ 
and $f(ab)=f(a)f(b)$ for any $\lambda \in \Lambda$ and $a,b \in G_\lambda$.
We call a bijective MCQ homomorphism an \textit{MCQ isomorphism}.
$X_1$ and $X_2$ are \textit{isomorphic}, denoted by $X_1 \cong X_2$, 
if there exists an MCQ isomorphism from $X_1$ to $X_2$.

For an MCQ $X=\bigsqcup_{\lambda \in \Lambda}G_\lambda$, 
an \textit{extension} of $X$ is an MCQ $\widetilde{X}$ 
which has a surjective MCQ homomorphism $f: \widetilde{X} \to X$
such that 
for any element of $X$, 
the cardinality of the inverse image by $f$ is constant.


Next, 
we recall the notion of a $G$-family of quandles~\cite{IshiiIwakiriJangOshiro13}.
It is an algebraic system which gives an MCQ.
Let $G$ be a group with the identity $e$.
A \textit{$G$-family of quandles} is a non-empty set $X$ 
with a family of binary operations $ \tri ^g:X \times X \to X ~(g \in G)$ 
satisfying the following axioms: 
\begin{itemize}
\item
For any $x \in X$ and $g \in G$, 
$x \tri ^gx=x.$
\item
For any $x,y \in X$ and $g,h \in G$, 
$x \tri ^ey=x$ and $x \tri ^{gh}y=(x \tri ^gy) \tri ^hy$.
\item
For any $x,y,z \in X$ and $g,h \in G$, 
$(x \tri ^gy) \tri ^hz=(x \tri ^hz) \tri ^{h^{-1}gh}(y \tri ^hz)$.
\end{itemize}

Let $R$ be a ring and $G$ be a group with the identity $e$.
Let $X$ be a right $R[G]$-module, where $R[G]$ is the group ring of $G$ over $R$.
Then $(X,\{  \tri ^g \}_{g \in G})$ is a $G$-family of quandles, called a \textit{$G$-family of Alexander quandles}, 
with $x \tri ^gy=xg+y(e-g)$.
Let $(Q, \tri )$ be a quandle.
Then $(Q,\{ \tri ^i\}_{i \in \mathbb{Z}})$ is a $\mathbb{Z}$-family of quandles.
Furthermore 
if $\type Q$ is finite, 
$(Q,\{ \tri ^i\}_{i \in \mathbb{Z}_{\type Q}})$ is a $\mathbb{Z}_{\type Q}$-family of quandles.

Let $(X,\{  \tri ^g \}_{g \in G})$ be a $G$-family of quandles.
Then $G \times X =\bigsqcup_{x \in X} (G \times \{ x \})$ is an MCQ with 
\begin{align*}
(g,x) \tri (h,y):=(h^{-1}gh,x \tri ^hy),
\hspace{7mm}
(g,x)(h,x):=(gh,x)
\end{align*}
for any $x,y \in X$ and $g,h \in G$~\cite{Ishii15}.
We call it the \textit{associated MCQ} 
of  $(X,\{  \tri ^g \}_{g \in G})$.
The associated MCQ $G \times X$ of a $G$-family of quandles $X$ is an extension of $G$, 
where we regard $G$ as an MCQ with the conjugation operation.

Throughout this paper, unless otherwise stated, 
we assume that every ring has the multiplicative identity $1 \neq 0$.
For a ring $R$, 
we denote by $R^\times$ the group of units of $R$.
In the following, 
we introduce a pair of maps, called an MCQ Alexander pair~\cite{Murao21}, 
which corresponds to a linear extension of an MCQ.
Furthermore 
we introduce a quadruple of maps, called an augmented MCQ Alexander pair, 
which corresponds to an affine extension of an MCQ 
as seen in Proposition~\ref{prop:augmented MCQ Alexander pair}.

\begin{definition}
Let $X=\bigsqcup_{\lambda \in \Lambda}G_\lambda$ be an MCQ, 
$R$ a ring and $M$ a left $R$-module.
Let $f_1,f_2 : X \times X \to R$, 
$\phi_1: X \times X \to M$ 
and $\phi_2: \bigsqcup_{\lambda \in \Lambda}(G_\lambda \times G_\lambda) \to M$ 
be maps.

\begin{enumerate}
\item
The pair $(f_1,f_2)$ is an \textit{MCQ Alexander pair} 
if $f_1$ and $f_2$ satisfy the following conditions: 

\begin{itemize} 
\item
For any $a,b \in G_\lambda$,
\begin{align*}
f_1(a,b)+f_2(a,b)=f_1(a,a^{-1}b).
\end{align*}

\item
For any $a,b \in G_\lambda$ and $x \in X$, 
\begin{align*}
&f_1(a,x)=f_1(b,x),\\
&f_2(ab,x)=f_2(a,x)+f_1(b \tri x,a^{-1} \tri x)f_2(b,x).
\end{align*}

\item
For any $x \in X$ and $a,b \in G_{\lambda}$,
\begin{align*}
&f_1(x,e_{\lambda})=1,\\
&f_1(x,ab)=f_1(x \tri a,b)f_1(x,a),\\
&f_2(x,ab)=f_1(x \tri a,b)f_2(x,a).
\end{align*}

\item
For any $x,y,z \in X$, 
\begin{align*}
&f_1(x \tri y,z)f_1(x,y)=f_1(x \tri z,y \tri z)f_1(x,z),\\
&f_1(x \tri y,z)f_2(x,y)=f_2(x \tri z,y \tri z)f_1(y,z),\\
&f_2(x \tri y,z)=f_1(x \tri z,y \tri z)f_2(x,z)+f_2(x \tri z,y \tri z)f_2(y,z).
\end{align*}
\end{itemize}

\item
For an MCQ Alexander pair $(f_1,f_2)$, 
the pair $(\phi_1,\phi_2)$ is an \textit{$(f_1,f_2)$-twisted 2-cocycle} 
if $\phi_1$ and $\phi_2$ satisfy the following conditions:

\begin{itemize}
\item
For any $a,b,c \in G_\lambda$, 
\[\phi_2(a,b)+\phi_2(ab,c)=f_1(a,a^{-1})\phi_2(b,c)+\phi_2(a,bc).\]

\item
For any $a,b \in G_\lambda$, 
\[f_1(b,b^{-1})\phi_1(a,b)+\phi_2(b,b^{-1}ab)=\phi_2(a,b).\]


\item
For any $x\in X$ and $a,b\in G_\lambda$, 
\[f_2(x,ab)\phi_2(a,b)+\phi_1(x,ab)=f_1(x \tri a,b)\phi_1(x,a)+\phi_1(x \tri a,b).\]

\item
For any $x,y,z \in X$, 
\begin{align*}
&f_1(x \tri y,z)\phi_1(x,y)+\phi_1(x\tri y,z)\\
&=f_1(x \tri z,y \tri z)\phi_1(x,z)+f_2(x \tri z,y \tri z)\phi_1(y,z)+\phi_1(x \tri z,y \tri z).
\end{align*}

\item
For any $a,b \in G_\lambda$ and $x\in X$, 
\begin{align*}
&f_1(ab,x)\phi_2(a,b)+\phi_1(ab,x)\\
&=\phi_1(a,x)+f_1(a \tri x,a^{-1} \tri x)\phi_1(b,x)+\phi_2(a \tri x,b \tri x).
\end{align*}
\end{itemize}
\end{enumerate}
We call $(f_1,f_2;\phi_1,\phi_2)$ an \textit{augmented MCQ Alexander pair} 
if $(f_1, f_2)$ is an MCQ Alexander pair and $(\phi_1,\phi_2)$ is an $(f_1, f_2)$-twisted 2-cocycle.
\end{definition}

By the definition and \cite[Lemma~2.6]{Murao21}, 
we have the following lemma immediately.

\begin{lemma}\label{lem:augmentedMCQAlexanderPairProperty}
Let $X=\bigsqcup_{\lambda \in \Lambda}G_\lambda$ be an MCQ, $R$ a ring 
and $M$ be a left $R$-module.
Let $(f_1,f_2)$ be an MCQ Alexander pair of maps $f_1,f_2 : X \times X \to R$ 
and let $(\phi_1,\phi_2)$ be an $(f_1,f_2)$-twisted 2-cocycle of maps 
$\phi_1: X \times X \to M$ 
and $\phi_2: \bigsqcup_{\lambda \in \Lambda}(G_\lambda \times G_\lambda) \to M$.
For any $x,y \in X$ and $a,b \in G_\lambda$, 
the following hold.
\begin{align*}
&\text{$f_1(x,y)$ is invertible, and~} 
f_1(x,y)^{-1}=f_1(x \tri y,y^{-1}),\\
&f_2(e_\lambda,x)=0,\\
&f_1(ab,x)f_1(a,a^{-1})=f_1(b \tri x,a^{-1} \tri x)f_1(b,x),\\
&f_2(x \tri a,b)=f_2(x,ab)f_1(a,a^{-1}),\\
&\phi_1(x,x)=0,\\
&\phi_2(e_\lambda,a)=\phi_2(e_\lambda,b).
\end{align*}
\end{lemma}

We call $(1,0)$ the \textit{trivial MCQ Alexander pair} 
and $(0,0)$ the \textit{trivial $(f_1, f_2)$-twisted 2-cocycle}, 
where $0$ and $1$ respectively denote the zero map and the constant map 
that sends all elements of the domain to the multiplicative identity $1$ of the ring.
The notion of a $(1,0)$-twisted 2-cocycle $(\phi_1,\phi_2)$ coincides with 
that of an MCQ 2-cocycle.
More precisely, 
we set the map 
\[\phi : R\left[(X \times X) \sqcup \left(\bigsqcup_{\lambda \in \Lambda}(G_\lambda \times G_\lambda) \right)\right] \to M\]
by 
\begin{align*}
\phi(x,y):=
\begin{cases}
\phi_1(x,y) &((x,y) \in X \times X),\\
\phi_2(x,y) &((x,y) \in \bigsqcup_{\lambda \in \Lambda}(G_\lambda \times G_\lambda))
\end{cases}
\end{align*}
and extend it linearly.
Then the pair $(\phi_1,\phi_2)$ is a $(1,0)$-twisted 2-cocycle 
if and only if 
$\phi$ is an MCQ 2-cocycle.
For more details, 
we refer the reader to \cite{CarterIshiiSaitoTanaka17}.
By the definition and Lemma~\ref{lem:augmentedMCQAlexanderPairProperty}, 
if $(f_1,f_2;\phi_1,\phi_2)$ is an augmented MCQ Alexander pair, 
then $(f_1,f_2;\phi_1)$ is an augmented Alexander pair~\cite{IshiiOshiro19}.
An augmented MCQ Alexander pair corresponds to an extension of an MCQ 
as shown in the following proposition.

\begin{proposition}\label{prop:augmented MCQ Alexander pair}
Let $X=\bigsqcup_{\lambda \in \Lambda}G_\lambda$ be an MCQ, 
$R$ a ring and $M$ a left $R$-module.
Let $f_1,f_2: X \times X \to R$, $\phi_1: X \times X \to M$ 
and $\phi_2: \bigsqcup_{\lambda \in \Lambda}(G_\lambda \times G_\lambda) \to M$ be maps.
If $(f_1,f_2; \phi_1,\phi_2)$ is an augmented MCQ Alexander pair, 
then
$\widetilde{X}(f_1,f_2;\phi_1,\phi_2):=\bigsqcup_{\lambda \in \Lambda}(G_\lambda \times M)$ 
is an MCQ with 
\begin{align*}
&(x,u) \tri (y,v):=(x \tri y, f_1(x,y)u+f_2(x,y)v+\phi_1(x,y)),\\
&(a,u)(b,v):=(ab,u+f_1(a,a^{-1})v+\phi_2(a,b))
\end{align*}
for any $(x,u),(y,v) \in \widetilde{X}(f_1,f_2;\phi_1,\phi_2)$ 
and $(a,u),(b,v) \in G_\lambda \times M$, 
where 
the identity of each group $G_\lambda \times M$ is $(e_\lambda,-\phi_2(e_\lambda,e_\lambda))$,
and the inverse of $(a,u) \in G_\lambda \times M$ is 
$(a^{-1},-f_1(a,a)u-\phi_2(a^{-1},a)-\phi_2(e_\lambda,e_\lambda))$.
Furthermore, the converse is true when $M=R$.
\end{proposition}

\begin{proof}
If $(f_1,f_2;\phi_1,\phi_2)$ is an augmented MCQ Alexander pair, 
then we have that 
$\widetilde{X}(f_1,f_2;\phi_1,\phi_2)=\bigsqcup_{\lambda \in \Lambda}(G_\lambda \times M)$ 
is an MCQ 
by direct calculation and by Lemma~\ref{lem:augmentedMCQAlexanderPairProperty}.
Furthermore, 
for any $(a,u) \in G_\lambda \times M$, 
it follows that 
\begin{align*}
&(a,u)(e_\lambda,-\phi_2(e_\lambda,e_\lambda))
=(a,u-f_1(a,a^{-1})\phi_2(e_\lambda,e_\lambda)+\phi_2(a,e_\lambda))
=(a,u),\\
&(e_\lambda,-\phi_2(e_\lambda,e_\lambda))(a,u)
=(a,-\phi_2(e_\lambda,e_\lambda)+f_1(e_\lambda,e_\lambda)u+\phi_2(e_\lambda,a))
=(a,u)
\end{align*}
and
\begin{align*}
&(a,u)(a^{-1},-f_1(a,a)u-\phi_2(a^{-1},a)-\phi_2(e_\lambda,e_\lambda))\\
&=(e_\lambda, u+f_1(a,a^{-1})(-f_1(a,a)u-\phi_2(a^{-1},a)-\phi_2(e_\lambda,e_\lambda))+\phi_2(a,a^{-1}))\\
&=(e_\lambda,f_1(a,a^{-1})(-f_1(a^{-1},a)\phi_2(a,e_\lambda)-\phi_2(a^{-1},a))+\phi_2(a,a^{-1}))\\
&=(e_\lambda,-\phi_2(a,e_\lambda)-f_1(a,a^{-1})\phi_2(a^{-1},a)+\phi_2(a,a^{-1}))\\
&=(e_\lambda, -\phi_2(e_\lambda,e_\lambda)),\\
&(a^{-1},-f_1(a,a)u-\phi_2(a^{-1},a)-\phi_2(e_\lambda,e_\lambda))(a,u)\\
&=(e_\lambda, -f_1(a,a)u-\phi_2(a^{-1},a)-\phi_2(e_\lambda,e_\lambda)+f_1(a^{-1},a)u+\phi_2(a^{-1},a))\\
&=(e_\lambda, -\phi_2(e_\lambda,e_\lambda)),
\end{align*}
which imply that 
the identity of $G_\lambda \times M$ is $(e_\lambda,-\phi_2(e_\lambda,e_\lambda))$, 
and the inverse of $(a,u) \in G_\lambda \times M$ is $(a^{-1},-f_1(a,a)u-\phi_2(a,a^{-1})-\phi_2(e_\lambda,e_\lambda))$.

Put $M:=R$.
Assume that $\widetilde{X}(f_1,f_2;\phi_1,\phi_2)=\bigsqcup_{\lambda \in \Lambda}(G_\lambda \times M)$ is an MCQ.
We prove that $(f_1,f_2;\phi_1,\phi_2)$ is an augmented MCQ Alexander pair.
For each $\lambda \in \Lambda$, $G_\lambda \times M$ is a group.
Hence for any $(a,u), (b,v), (c,w) \in G_\lambda \times M$, it follows that
\begin{align*}
&((a,u)(b,v))(c,w)\\
&= (ab,u+f_1(a,a^{-1})v+\phi_2(a,b))(c,w)\\
&= (abc,u+f_1(a,a^{-1})v+\phi_2(a,b)+f_1(ab,b^{-1}a^{-1})w+\phi_2(ab,c)),\\
&(a,u)((b,v)(c,w))\\
&= (a,u)(bc,v+f_1(b,b^{-1})w+\phi_2(b,c))\\
&= (abc,u+f_1(a,a^{-1})(v+f_1(b,b^{-1})w+\phi_2(b,c))+\phi_2(a,bc)).
\end{align*}
By the associativity of $G_\lambda \times M$, 
we have that for any $a,b,c \in G_\lambda$, 
\begin{align}
&f_1(a,a^{-1})f_1(b,b^{-1})=f_1(ab,b^{-1}a^{-1}),\\
&\phi_2(a,b)+\phi_2(ab,c)=f_1(a,a^{-1})\phi_2(b,c)+\phi_2(a,bc).\tag{$\phi$-1}
\end{align}

For any $(a,u), (b,v) \in G_\lambda \times M$, $(a,u) \tri (b,v)=(b,v)^{-1}(a,u)(b,v)$.
It follows that 
\begin{align*}
&(a,u) \tri (b,v)
= (b^{-1}ab,f_1(a,b)u+f_2(a,b)v+\phi_1(a,b)),\\
&(b,v)^{-1}(a,u)(b,v)\\
&= (b^{-1},-f_1(b,b)v-\phi_2(b^{-1},b)-\phi_2(e_\lambda,e_\lambda))(ab,u+f_1(a,a^{-1})v+\phi_2(a,b))\\
&=(b^{-1}ab, -f_1(b,b)v-\phi_2(b^{-1},b)-\phi_2(e_\lambda,e_\lambda)\\
&\quad +f_1(b^{-1},b)(u+f_1(a,a^{-1})v+\phi_2(a,b))+\phi_2(b^{-1},ab))
\end{align*}
Hence we have that for any $a,b \in G_\lambda$, 
\begin{align}
&f_1(a,b)=f_1(b^{-1},b), \notag\\
&f_2(a,b)=-f_1(b,b)+f_1(b^{-1},b)f_1(a,a^{-1}),\\
&\phi_1(a,b)=-\phi_2(b^{-1},b)-\phi_2(e_\lambda,e_\lambda)+f_1(b^{-1},b)\phi_2(a,b)+\phi_2(b^{-1},ab).\tag{$\phi$-2}
\end{align}

For any $(x,u) \in \widetilde{X}(f_1,f_2;\phi_1,\phi_2)$ and $(a,v), (b,w) \in G_{\lambda} \times M$, 
$(x,u) \tri (e_{\lambda},-\phi_2(e_\lambda,e_\lambda))=(x,u)$ 
and $(x,u) \tri ((a,v)(b,w))=((x,u) \tri (a,v)) \tri (b,w)$.
It follows that 
\begin{align*}
&(x,u) \tri (e_{\lambda},-\phi_2(e_\lambda,e_\lambda))\\
&=(x,f_1(x,e_{\lambda})u-f_2(x,e_{\lambda})\phi_2(e_\lambda,e_\lambda)+\phi_1(x,e_\lambda)),\\
&(x,u) \tri ((a,v)(b,w))\\
&= (x,u) \tri (ab,v+f_1(a,a^{-1})w+\phi_2(a,b))\\
&= (x \tri ab,f_1(x,ab)u+f_2(x,ab)(v+f_1(a,a^{-1})w+\phi_2(a,b))+\phi_1(x,ab)),\\
&((x,u) \tri (a,v)) \tri (b,w)\\
&= (x \tri a,f_1(x,a)u+f_2(x,a)v+\phi_1(x,a)) \tri (b,w)\\
&= ((x \tri a) \tri b,f_1(x \tri a,b)(f_1(x,a)u+f_2(x,a)v+\phi_1(x,a))\\
&\quad+f_2(x \tri a,b)w+\phi_1(x \tri a,b)).
\end{align*}
Hence we have that for any $x \in X$ and $a,b \in G_{\lambda}$, 
\begin{align}
&f_1(x,e_{\lambda}) =1,\\
&f_1(x,ab) = f_1(x \tri a,b)f_1(x,a), \\ 
&f_2(x,ab) = f_1(x \tri a,b)f_2(x,a),\\
&f_2(x,ab)f_1(a,a^{-1})=f_2(x \tri a,b), \notag\\
&f_2(x,e_{\lambda})\phi_2(e_\lambda,e_\lambda)=\phi_1(x,e_\lambda), \notag\\
&f_2(x,ab)\phi_2(a,b)+\phi_1(x,ab)=f_1(x \tri a,b)\phi_1(x,a)+\phi_1(x \tri a,b). \tag{$\phi$-3}
\end{align}

For any $(x,u), (y,v), (z,w) \in \widetilde{X}(f_1,f_2;\phi_1,\phi_2)$, 
$((x,u) \tri (y,v)) \tri (z,w)=((x,u) \tri (z,w)) \tri ((y,v) \tri (z,w))$.
It follows that 
\begin{align*}
& ((x,u) \tri (y,v)) \tri (z,w) \\
&= (x \tri y,f_1(x,y)u+f_2(x,y)v+\phi_1(x,y)) \tri (z,w)\\
&= ((x \tri y) \tri z,f_1(x \tri y,z)(f_1(x,y)u+f_2(x,y)v+\phi_1(x,y))\\
&\quad+f_2(x \tri y,z)w+\phi_1(x \tri y,z)),\\
& ((x,u) \tri (z,w)) \tri ((y,v) \tri (z,w))\\
&=(x \tri z, f_1(x,z)u+f_2(x,z)w+\phi_1(x,z)) \tri (y \tri z, f_1(y,z)v+f_2(y,z)w+\phi_1(y,z))\\
&=((x \tri z)\tri (y \tri z), f_1(x \tri z, y \tri z)( f_1(x,z)u+f_2(x,z)w+\phi_1(x,z))\\
& \quad +f_2(x \tri z, y \tri z)(f_1(y,z)v+f_2(y,z)w+\phi_1(y,z))+\phi_1(x \tri z, y \tri z)).
\end{align*}
Hence we have that for any $x,y,z \in X$, 
\begin{align}
&f_1(x \tri y,z)f_1(x,y) = f_1(x \tri z, y \tri z)f_1(x,z), \\
&f_1(x \tri y,z)f_2(x,y) = f_2(x \tri z, y \tri z)f_1(y,z), \\
&f_2(x \tri y,z) =  f_1(x \tri z, y \tri z)f_2(x,z)+f_2(x \tri z, y \tri z)f_2(y,z),\\
&f_1(x \tri y,z)\phi_1(x,y)+\phi_1(x \tri y,z) \notag\\
&=f_1(x\tri z,y\tri z)\phi_1(x,z)+f_2(x\tri z,y\tri z)\phi_1(y,z)+\phi_1(x \tri z,y\tri z).\tag{$\phi$-4}
\end{align}

For any $(a,u), (b,v) \in G_\lambda \times M$ 
and $(x,w) \in \widetilde{X}(f_1,f_2;\phi_1,\phi_2)$, 
$((a,u)(b,v)) \tri (x,w)=((a,u) \tri (x,w))((b,v) \tri (x,w))$, 
where we note that 
$(a,u) \tri (x,w), (b,v) \tri (x,w) \in G_{\mu} \times M$ for some $\mu \in \Lambda$.
It follows that 
\begin{align*}
&((a,u)(b,v)) \tri (x,w)\\
&=(ab,u+f_1(a,a^{-1})v+\phi_2(a,b)) \tri (x,w)\\
&=(ab \tri x,f_1(ab,x)(u+f_1(a,a^{-1})v+\phi_2(a,b)) + f_2(ab,x)w+\phi_1(ab,x)),\\
&((a,u) \tri (x,w))((b,v) \tri (x,w)) \\
&=(a \tri x, f_1(a,x)u+f_2(a,x)w+\phi_1(a,x))(b \tri x, f_1(b,x)v+f_2(b,x)w+\phi_1(b,x))\\
&=((a \tri x)(b \tri x),f_1(a,x)u+f_2(a,x)w+\phi_1(a,x)\\
&\quad +f_1(a \tri x,a^{-1} \tri x)(f_1(b,x)v+f_2(b,x)w+\phi_1(b,x))+\phi_2(a \tri x,b \tri x)).
\end{align*}
Hence we have that for any $a,b \in G_\lambda$ and $x \in X$, 
\begin{align}
&f_1(ab,x)=f_1(a,x),\\
&f_1(ab,x)f_1(a,a^{-1})=f_1(a \tri x,a^{-1} \tri x)f_1(b,x), \notag \\
&f_2(ab,x)=f_2(a,x)+f_1(a \tri x,a^{-1} \tri x)f_2(b,x),\\
&f_1(ab,x)\phi_2(a,b)+\phi_1(ab,x) \notag\\
&=\phi_1(a,x)+f_1(a \tri x,a^{-1} \tri x)\phi_1(b,x)+\phi_2(a\tri x,b\tri x).\tag{$\phi$-5}
\end{align}

By equations (1), (2), (9) and (10), 
we have that for any $a,b \in G_\lambda$ and $x \in X$, 
\begin{align}
&f_1(a,b)+f_2(a,b)=f_1(a,a^{-1}b), \tag{2'}\\
&f_1(a,x)=f_1(b,x), \tag{9'}\\
&f_2(ab,x)=f_2(a,x)+f_1(b \tri x,a^{-1} \tri x)f_2(b,x). \tag{10'}
\end{align}
Therefore, 
by the equations (2'), (3)--(8), (9') and (10'), 
the pair $(f_1,f_2)$ is an MCQ Alexander pair.
Moreover, 
by equations ($\phi$-1) and ($\phi$-2), 
we have that for any $a,b \in G_\lambda$, 
\begin{align*}
\phi_1(a,b)
&=-\phi_2(b^{-1},b)-\phi_2(e_\lambda,e_\lambda)+f_1(b^{-1},b)\phi_2(a,b)+\phi_2(b^{-1},ab)\\
&=-f_1(b^{-1},b)\phi_2(b,b^{-1}ab)+f_1(b^{-1},b)\phi_2(a,b),
\end{align*}
which implies 
\begin{align}
f_1(b,b^{-1})\phi_1(a,b)+\phi_2(b,b^{-1}ab)=\phi_2(a,b).\tag{$\phi$-2'}
\end{align}
Therefore, 
by the equations ($\phi$-1), ($\phi$-2') and ($\phi$-3)--($\phi$-5), 
the pair $(\phi_1,\phi_2)$ is an $(f_1,f_2)$-twisted 2-cocycle.
\end{proof}

We remark that 
the MCQ $\widetilde{X}(f_1,f_2;\phi_1,\phi_2)=\bigsqcup_{\lambda \in \Lambda}(G_\lambda \times M)$ 
in Proposition~\ref{prop:augmented MCQ Alexander pair} is an extension of $X$ 
since the projection from $\widetilde{X}(f_1,f_2;\phi_1,\phi_2)$ to $X$ sending $(x,u)$ into $x$ 
satisfies the defining condition of an extension.

\section{Affine extensions of multiple conjugation quandles}
\label{sec:Affine extensions of multiple conjugation quandles}

Let $X=\bigsqcup_{\lambda \in \Lambda}G_\lambda$ be an MCQ, $R$ a ring and 
$M$ be a left $R$-module. 
Let $f_1,f_2: X \times X \to R$, 
$f_3,f_4: \bigsqcup_{\lambda \in \Lambda}(G_\lambda \times G_\lambda) \to R$,  
$\phi_1: X \times X \to M$ and $\phi_2: \bigsqcup_{\lambda \in \Lambda}(G_\lambda \times G_\lambda) \to M$ 
be maps.
In this section, 
we consider an affine extension of $X$ using $f_1, f_2, f_3, f_4, \phi_1$ and $\phi_2$.

We define the conditions  \eqref{eq:0-i}--\eqref{eq:4-phi} 
for $f_1, f_2, f_3, f_4, \phi_1$ and $\phi_2$ 
as follows:

\begin{itemize}
\item
For any $a,b,c \in G_\lambda$, 
\begin{align}
& \text{$f_3(a,b)$ and $f_4(a,b)$ are invertible}, \tag{0-i} \label{eq:0-i}\\
& f_3(ab,c)f_3(a,b)=f_3(a,bc),\tag{0-ii} \label{eq:0-ii}\\
& f_3(ab,c)f_4(a,b)=f_4(a,bc)f_3(b,c),\tag{0-iii} \label{eq:0-iii}\\
& f_4(ab,c)=f_4(a,bc)f_4(b,c)\tag{0-iv}, \label{eq:0-iv}\\
& f_3(ab,c)\phi_2(a,b)+\phi_2(ab,c)=f_4(a,bc)\phi_2(b,c)+\phi_2(a,bc)\tag{0-$\phi$}. \label{eq:0-phi}
\end{align}

\item
For any $a,b \in G_\lambda$, 
\begin{align*}
&f_1(a,b)=f_4(b^{-1},ab)f_3(a,b),\tag{1-i} \label{eq:1-i}\\
&f_3(b,b^{-1}ab)+f_4(b,b^{-1}ab)f_2(a,b)=f_4(a,b), \tag{1-ii} \label{eq:1-ii}\\
&f_4(b,b^{-1}ab)\phi_1(a,b)+\phi_2(b,b^{-1}ab)=\phi_2(a,b). \tag{1-$\phi$} \label{eq:1-phi}
\end{align*}

\item
For any $x \in X$ and $a,b \in G_{\lambda}$, 
\begin{align*}
&f_1(x,e_{\lambda}) =1,\tag{2-i} \label{eq:2-i}\\
&f_1(x,ab) = f_1(x \tri a,b)f_1(x,a), \tag{2-ii} \label{eq:2-ii}\\ 
&f_2(x,ab)f_3(a,b) = f_1(x \tri a,b)f_2(x,a), \tag{2-iii} \label{eq:2-iii}\\
&f_2(x,ab)f_4(a,b)=f_2(x \tri a,b), \tag{2-iv} \label{eq:2-iv}\\
&f_2(x,e_\lambda)\phi_2(e_\lambda,e_\lambda)=\phi_1(x,e_\lambda), \tag{2-$\phi$i} \label{eq:2-phii}\\
&f_2(x,ab)\phi_2(a,b)+\phi_1(x,ab)=f_1(x \tri a,b)\phi_1(x,a)+\phi_1(x \tri a,b). \tag{2-$\phi$ii} \label{eq:2-phiii}\\
\end{align*}

\item
For any $x,y,z \in X$, 
\begin{align*}
&f_1(x \tri y,z)f_1(x,y) = f_1(x \tri z, y \tri z)f_1(x,z), \tag{3-i} \label{eq:3-i}\\
&f_1(x \tri y,z)f_2(x,y) = f_2(x \tri z, y \tri z)f_1(y,z), \tag{3-ii} \label{eq:3-ii}\\
&f_2(x \tri y,z) =  f_1(x \tri z, y \tri z)f_2(x,z)+f_2(x \tri z, y \tri z)f_2(y,z), \tag{3-iii} \label{eq:3-iii}\\
&f_1(x \tri y,z)\phi_1(x,y)+\phi_1(x\tri y,z)\\
&=f_1(x \tri z,y \tri z)\phi_1(x,z)+f_2(x \tri z,y \tri z)\phi_1(y,z)+\phi_1(x \tri z,y \tri z). \tag{3-$\phi$} \label{eq:3-phi}
\end{align*}

\item
For any $a,b \in G_\lambda$ and $x \in X$, 
\begin{align*}
&f_1(ab,x)f_3(a,b)=f_3(a \tri x,b \tri x)f_1(a,x), \tag{4-i} \label{eq:4-i}\\
&f_1(ab,x)f_4(a,b)=f_4(a \tri x,b \tri x)f_1(b,x), \tag{4-ii} \label{eq:4-ii}\\
&f_2(ab,x)=f_3(a \tri x,b \tri x)f_2(a,x)+f_4(a \tri x,b \tri x)f_2(b,x), \tag{4-iii} \label{eq:4-iii}\\
&f_1(ab,x)\phi_2(a,b)+\phi_1(ab,x)\\
& \quad= f_3(a \tri x,b \tri x)\phi_1(a,x)+f_4(a \tri x,b \tri x)\phi_1(b,x)+\phi_2(a \tri x,b \tri x). \tag{4-$\phi$} \label{eq:4-phi}
\end{align*}
\end{itemize}

These conditions correspond to every affine extension of an MCQ 
as seen in Proposition~\ref{Alexander 6-tuple}.
We remark that for any $a,b \in G_\lambda$, 
it follows 
\begin{align*}
&f_3(a,e_\lambda)=1=f_4(e_\lambda,a),\\
&f_3(a,b)^{-1}=f_3(ab,b^{-1}),\\
&f_4(a,b)^{-1}=f_4(a^{-1},ab),\\
&\phi_1(a,a)=0
\end{align*}
by \eqref{eq:0-i}, \eqref{eq:0-ii}, \eqref{eq:0-iv} and \eqref{eq:1-phi}.

\begin{remark}\label{replace condition}
The condition \eqref{eq:1-ii} can be replaced with the following condition: 
\begin{align*}
f_2(a,b) &=-f_3(b^{-1},ab)f_4(b^{-1},e_\lambda)f_3(b,b^{-1})+f_4(b^{-1},ab)f_4(a,b),
\end{align*} 
which was used in~\cite{Murao21} instead of the condition \eqref{eq:1-ii}.
\end{remark}

\begin{proposition}\label{Alexander 6-tuple}
Let $X=\bigsqcup_{\lambda \in \Lambda}G_\lambda$ be an MCQ, $R$ a ring and 
$M$ be a left $R$-module. 
Let $f_1,f_2: X \times X \to R$, 
$f_3,f_4: \bigsqcup_{\lambda \in \Lambda}(G_\lambda \times G_\lambda) \to R$, 
$\phi_1: X \times X \to M$ and 
$\phi_2: \bigsqcup_{\lambda \in \Lambda}(G_\lambda \times G_\lambda) \to M$ be maps.
If $f_1$, $f_2$, $f_3$, $f_4$, $\phi_1$ and $\phi_2$ satisfy the conditions \eqref{eq:0-i}--\eqref{eq:4-phi}, 
then 
$\widetilde{X}(f_1,f_2,f_3,f_4;\phi_1, \phi_2):=\bigsqcup_{\lambda \in \Lambda}(G_\lambda \times M)$ 
is an MCQ with 
\begin{align*}
& (x,u) \tri (y,v):=(x \tri y, f_1(x,y)u+f_2(x,y)v+\phi_1(x,y))
,\\
& (a,u)(b,v):=(ab,f_3(a,b)u+f_4(a,b)v+\phi_2(a,b))
\end{align*}
for any $(x,u),(y,v) \in \widetilde{X}(f_1,f_2,f_3,f_4;\phi_1, \phi_2)$ 
and $(a,u),(b,v) \in G_\lambda \times M$, 
where 
the identity of each group $G_\lambda \times M$ is 
$(e_\lambda,-\phi_2(e_\lambda,e_\lambda))$, 
and the inverse of $(a,u) \in G_\lambda \times M$ is 
$(a^{-1},-f_3(e_\lambda,a^{-1})(f_4(a^{-1},a)u+\phi_2(a^{-1},a)+\phi_2(e_\lambda,e_\lambda)))$.
Furthermore, the converse is true when $M=R$.
\end{proposition}

Let us first prove the following lemma 
in order to prove Proposition~\ref{Alexander 6-tuple} later.

\begin{lemma}\label{group condition}
In the same situation as Proposition~\ref{Alexander 6-tuple}, 
if the maps $f_3, f_4$ and $\phi_2$ satisfy the conditions \eqref{eq:0-i}--\eqref{eq:0-phi}, 
then $G_\lambda \times M$ is a group for each $\lambda \in \Lambda$, 
where 
the identity of each group $G_\lambda \times M$ is 
$(e_\lambda,-\phi_2(e_\lambda,e_\lambda))$, 
and the inverse of $(a,u) \in G_\lambda \times M$ is 
$(a^{-1},-f_3(e_\lambda,a^{-1})(f_4(a^{-1},a)u+\phi_2(a^{-1},a)+\phi_2(e_\lambda,e_\lambda)))$.
Furthermore, the converse is true when $M=R$.
\end{lemma}

\begin{proof}
If the maps $f_3, f_4$ and $\phi_2$ satisfy the conditions \eqref{eq:0-i}--\eqref{eq:0-phi}, 
then we have that $G_\lambda \times M$ is a group for each $\lambda \in \Lambda$ 
by direct calculation.
Furthermore, 
for any $(a,u) \in G_\lambda \times M$, 
it follows that 
\begin{align*}
(a,u)(e_\lambda,-\phi_2(e_\lambda,e_\lambda))
&=(a,f_3(a,e_\lambda)u-f_4(a,e_\lambda)\phi_2(e_\lambda,e_\lambda)+\phi_2(a,e_\lambda))\\
&=(a,u-f_3(a,e_\lambda)\phi_2(a,e_\lambda)+\phi_2(a,e_\lambda))\\
&=(a,u),\\
(e_\lambda,-\phi_2(e_\lambda,e_\lambda))(a,u)
&=(a,-f_3(e_\lambda,a)\phi_2(e_\lambda,e_\lambda)+f_4(e_\lambda,a)u+\phi_2(e_\lambda,a))\\
&=(a,-f_4(e_\lambda,a)\phi_2(e_\lambda,a)+u+\phi_2(e_\lambda,a))\\
&=(a,u)
\end{align*}
and
\begin{align*}
&(a,u)(a^{-1},-f_3(e_\lambda,a^{-1})(f_4(a^{-1},a)u+\phi_2(a^{-1},a)+\phi_2(e_\lambda,e_\lambda)))\\
&=(e_\lambda, f_3(a,a^{-1})u-f_4(a,a^{-1})f_3(e_\lambda,a^{-1})(f_4(a^{-1},a)u+\phi_2(a^{-1},a)+\phi_2(e_\lambda,e_\lambda))\\
&\quad+\phi_2(a,a^{-1}))\\
&=(e_\lambda, f_3(a,a^{-1})u-f_3(a,a^{-1})f_4(a,e_\lambda)f_4(a^{-1},a)u\\
&\quad-f_4(a,a^{-1})f_3(e_\lambda,a^{-1})\phi_2(a^{-1},a)-f_4(a,a^{-1})f_3(e_\lambda,a^{-1})\phi_2(e_\lambda,e_\lambda)+\phi_2(a,a^{-1}))\\
&=(e_\lambda, -f_4(a,a^{-1})f_3(e_\lambda,a^{-1})\phi_2(a^{-1},a)-f_4(a,a^{-1})\phi_2(e_\lambda,a^{-1})+\phi_2(a,a^{-1}))\\
&=(e_\lambda, -f_4(a,a^{-1})\phi_2(a^{-1},e_\lambda))\\
&=(e_\lambda, -\phi_2(e_\lambda,e_\lambda)),\\
&(a^{-1},-f_3(e_\lambda,a^{-1})(f_4(a^{-1},a)u+\phi_2(a^{-1},a)+\phi_2(e_\lambda,e_\lambda)))(a,u)\\
&=(e_\lambda,-f_3(a^{-1},a)f_3(e_\lambda,a^{-1})(f_4(a^{-1},a)u+\phi_2(a^{-1},a)+\phi_2(e_\lambda,e_\lambda))\\
&\quad+f_4(a^{-1},a)u+\phi_2(a^{-1},a))\\
&=(e_\lambda,-\phi_2(e_\lambda,e_\lambda)),
\end{align*}
which imply that 
the identity of $G_\lambda \times M$ is 
$(e_\lambda,-\phi_2(e_\lambda,e_\lambda))$, 
and the inverse of $(a,u) \in G_\lambda \times M$ is 
$(a^{-1},-f_3(e_\lambda,a^{-1})(f_4(a^{-1},a)u+\phi_2(a^{-1},a)+\phi_2(e_\lambda,e_\lambda)))$.

Put $M:=R$.
Assume that $G_\lambda \times M$ is a group for each $\lambda \in \Lambda$.
Then it follows that 
for any $(a,u),(b,v),(c,w) \in G_\lambda \times M$, 
\begin{align*}
&((a,u)(b,v))(c,w)\\
&= (ab,f_3(a,b)u+f_4(a,b)v+\phi_2(a,b))(c,w)\\
&= (abc,f_3(ab,c)(f_3(a,b)u+f_4(a,b)v+\phi_2(a,b))+f_4(ab,c)w+\phi_2(ab,c)),\\
&(a,u)((b,v)(c,w))\\
&= (a,u)(bc,f_3(b,c)v+f_4(b,c)w+\phi_2(b,c))\\
&= (abc,f_3(a,bc)u+f_4(a,bc)(f_3(b,c)v+f_4(b,c)w+\phi_2(b,c))+\phi_2(a,bc)).
\end{align*}
By the associativity of $G_\lambda \times M$, 
we have the conditions \eqref{eq:0-ii}, \eqref{eq:0-iii}, \eqref{eq:0-iv} and \eqref{eq:0-phi}.

Let $(g,m)$ be the identity of $G_\lambda \times M$.
Then for any $(a,u) \in G_\lambda \times M$, 
it follows that 
\begin{align*}
(a,u)(g,m)=(ag,f_3(a,g)u+f_4(a,g)m+\phi_2(a,g))=(a,u),\\
(g,m)(a,u)=(ga,f_3(g,a)m+f_4(g,a)u+\phi_2(g,a))=(a,u).
\end{align*}
Hence we have $g=e_\lambda$ and $f_3(a,e_\lambda)=1=f_4(e_\lambda,a)$ for any $a \in G_\lambda$.
By \eqref{eq:0-ii} and \eqref{eq:0-iv}, 
we obtain the condition \eqref{eq:0-i}.
\end{proof}

\begin{proof}[Proof of Proposition~\ref{Alexander 6-tuple}]
If $f_1, f_2, f_3, f_4, \phi_1$ and $\phi_2$ satisfy the conditions \eqref{eq:0-i}--\eqref{eq:4-phi}, 
then we have that 
$\widetilde{X}(f_1,f_2,f_3,f_4;\phi_1,\phi_2)$ is an MCQ 
by direct calculation and by Lemma~\ref{group condition}.

Put $M:=R$.
Assume that 
$\widetilde{X}(f_1,f_2,f_3,f_4;\phi_1,\phi_2)=\bigsqcup_{\lambda \in \Lambda}(G_\lambda \times M)$ is an MCQ.
For each $\lambda \in \Lambda$, 
$G_\lambda \times M$ is a group.
Hence we have the conditions \eqref{eq:0-i}--\eqref{eq:0-phi} 
by Lemma~\ref{group condition}.

For any $(a,u), (b,v) \in G_\lambda \times M$, $(a,u) \tri (b,v)=(b,v)^{-1}(a,u)(b,v)$.
It follows that 
\begin{align*}
&(b,v)((a,u) \tri (b,v))\\
&= (b,v)(b^{-1}ab,f_1(a,b)u+f_2(a,b)v+\phi_1(a,b))\\
&= (ab,f_3(b,b^{-1}ab)v+f_4(b,b^{-1}ab)(f_1(a,b)u+f_2(a,b)v+\phi_1(a,b))+\phi_2(b,b^{-1}ab)),\\
&(a,u)(b,v)
=(ab,f_3(a,b)u+f_4(a,b)v+\phi_2(a,b)).
\end{align*}
Hence we have that $f_1,f_2,f_3,f_4,\phi_1$ and $\phi_2$ 
satisfy the conditions \eqref{eq:1-i}--\eqref{eq:1-phi}.

For any $(x,u) \in \widetilde{X}(f_1,f_2,f_3,f_4;\phi_1,\phi_2)$ and $(a,v), (b,w) \in G_{\lambda} \times M$, 
$(x,u) \tri (e_\lambda,-\phi_2(e_\lambda,e_\lambda))=(x,u)$ and 
$(x,u) \tri ((a,v)(b,w))=((x,u) \tri (a,v)) \tri (b,w)$, 
where $(e_\lambda,-\phi_2(e_\lambda,e_\lambda))$ is the identity of $G_{\lambda} \times M$.
It follows that 
\begin{align*}
&(x,u) \tri (e_\lambda,-\phi_2(e_\lambda,e_\lambda))\\
&=(x,f_1(x,e_{\lambda})u-f_2(x,e_\lambda)\phi_2(e_\lambda,e_\lambda)+\phi_1(x,e_\lambda)),\\
&(x,u) \tri ((a,v)(b,w))\\
&= (x,u) \tri (ab,f_3(a,b)v+f_4(a,b)w+\phi_2(a,b))\\
&= (x \tri ab,f_1(x,ab)u+f_2(x,ab)(f_3(a,b)v+f_4(a,b)w+\phi_2(a,b))+\phi_1(x,ab)),\\
&((x,u) \tri (a,v)) \tri (b,w)\\
&= (x \tri a,f_1(x,a)u+f_2(x,a)v+\phi_1(x,a)) \tri (b,w)\\
&= ((x \tri a) \tri b,f_1(x \tri a,b)(f_1(x,a)u+f_2(x,a)v+\phi_1(x,a))\\
&\quad+f_2(x \tri a,b)w+\phi_1(x \tri a,b)).
\end{align*}
Hence we have that $f_1,f_2,f_3,f_4,\phi_1$ and $\phi_2$ 
satisfy the conditions \eqref{eq:2-i}--\eqref{eq:2-phiii}.

For any $(x,u), (y,v), (z,w) \in \widetilde{X}(f_1,f_2,f_3,f_4;\phi_1,\phi_2)$, 
$((x,u) \tri (y,v)) \tri (z,w)=((x,u) \tri (z,w)) \tri ((y,v) \tri (z,w))$.
It follows that 
\begin{align*}
& ((x,u) \tri (y,v)) \tri (z,w) \\
&= (x \tri y,f_1(x,y)u+f_2(x,y)v+\phi_1(x,y)) \tri (z,w)\\
&= ((x \tri y) \tri z,f_1(x \tri y,z)(f_1(x,y)u+f_2(x,y)v+\phi_1(x,y))\\
&\quad+f_2(x \tri y,z)w+\phi_1(x \tri y,z)),\\
& ((x,u) \tri (z,w)) \tri ((y,v) \tri (z,w))\\
&=(x \tri z, f_1(x,z)u+f_2(x,z)w+\phi_1(x,z)) \tri (y \tri z, f_1(y,z)v+f_2(y,z)w+\phi_1(y,z))\\
&=((x \tri z)\tri (y \tri z), f_1(x \tri z, y \tri z)( f_1(x,z)u+f_2(x,z)w+\phi_1(x,z))\\
& \quad +f_2(x \tri z, y \tri z)(f_1(y,z)v+f_2(y,z)w+\phi_1(y,z))+\phi_1(x \tri z,y \tri z)).
\end{align*}
Hence we have that $f_1,f_2,f_3,f_4,\phi_1$ and $\phi_2$ 
satisfy the conditions \eqref{eq:3-i}--\eqref{eq:3-phi}.

For any $(a,u), (b,v) \in G_\lambda \times M$ and $(x,w) \in \widetilde{X}(f_1,f_2,f_3,f_4;\phi_1,\phi_2)$, 
$((a,u)(b,v)) \tri (x,w)=((a,u) \tri (x,w))((b,v) \tri (x,w))$, 
where we note that 
$(a,u) \tri (x,w), (b,v) \tri (x,w) \in G_{\mu} \times M$ for some $\mu \in \Lambda$.
It follows that 
\begin{align*}
&((a,u)(b,v)) \tri (x,w)\\
&=(ab,f_3(a,b)u+f_4(a,b)v+\phi_2(a,b)) \tri (x,w)\\
&=(ab \tri x,f_1(ab,x)(f_3(a,b)u+f_4(a,b)v+\phi_2(a,b)) + f_2(ab,x)w+\phi_1(ab,x)),\\
&((a,u) \tri (x,w))((b,v) \tri (x,w)) \\
&=(a \tri x, f_1(a,x)u+f_2(a,x)w+\phi_1(a,x))(b \tri x, f_1(b,x)v+f_2(b,x)w+\phi_1(b,x))\\
&=((a \tri x)(b \tri x),f_3(a \tri x,b \tri x)(f_1(a,x)u+f_2(a,x)w+\phi_1(a,x))\\
&\quad +f_4(a \tri x,b \tri x)(f_1(b,x)v+f_2(b,x)w+\phi_1(b,x))+\phi_2(a \tri x,b \tri x).
\end{align*}
Hence we have that $f_1,f_2,f_3,f_4,\phi_1$ and $\phi_2$ 
satisfy the conditions \eqref{eq:4-i}--\eqref{eq:4-phi}.
\end{proof}

We remark that 
the MCQ $\widetilde{X}(f_1,f_2,f_3,f_4;\phi_1,\phi_2)=\bigsqcup_{\lambda \in \Lambda}(G_\lambda \times M)$ 
in Proposition~\ref{Alexander 6-tuple} 
is an extension of $X$ 
since the projection from $\widetilde{X}(f_1,f_2,f_3,f_4;\phi_1,\phi_2)$ to $X$ 
sending $(x,u)$ into $x$ 
satisfies the defining condition of an extension.
We call it an \textit{affine extension} of $X$.

\section{The reduction of affine extensions of MCQs to augmented MCQ Alexander pairs}
\label{sec:The reduction of affine extensions of MCQs to augmented MCQ Alexander pairs}

In this section, 
we see that 
any 6-tuple of maps satisfying the conditions \eqref{eq:0-i}--\eqref{eq:4-phi} 
can be reduced to some augmented MCQ Alexander pair.
That is, 
any affine extension of an MCQ can be realized by some augmented MCQ Alexander pair up to isomorphism.

Let $X=\bigsqcup_{\lambda \in \Lambda}G_\lambda$ be an MCQ, $R$ a ring and $M$ a left $R$-module.
Let $(f_1,f_2,f_3,f_4;\phi_1,\phi_2)$ and $(g_1,g_2,g_3,g_4;\psi_1,\psi_2)$ 
be 6-tuples of maps satisfying the conditions \eqref{eq:0-i}--\eqref{eq:4-phi}.
Then we write $(f_1,f_2,f_3,f_4;\phi_1,\phi_2) \sim (g_1,g_2,g_3,g_4;\psi_1,\psi_2)$ 
if there exist maps $h:X \to R^\times$ and $\eta : X \to M$ satisfying the following conditions:

\begin{itemize}
\item
For any $x,y \in X$, 
\begin{align*}
&h(x \tri y)f_1(x,y)= g_1(x,y)h(x),\\
&h(x \tri y)f_2(x,y)=g_2(x,y)h(y),\\
&h(x \tri y)\phi_1(x,y)+\eta(x \tri y)=g_1(x,y)\eta(x)+g_2(x,y)\eta(y)+\psi_1(x,y).
\end{align*}

\item
For any $a,b \in G_\lambda$, 
\begin{align*}
&h(ab)f_3(a,b)=g_3(a,b)h(a),\\
&h(ab)f_4(a,b)=g_4(a,b)h(b),\\
&h(ab)\phi_2(a,b)+\eta(ab)=g_3(a,b)\eta(a)+g_4(a,b)\eta(b)+\psi_2(a,b).
\end{align*}

\end{itemize}
Then $\sim$ is an equivalence relation 
on the set of all 6-tuples of maps satisfying the conditions \eqref{eq:0-i}--\eqref{eq:4-phi}.
We often write $(f_1,f_2,f_3,f_4;\phi_1,\phi_2) \sim_{h,\eta} (g_1,g_2,g_3,g_4;\psi_1,\psi_2)$ 
to specify $h$ and $\eta$.
This equivalence relation gives 
an isomorphic affine extensions of MCQs 
as seen in the following proposition.


\begin{proposition}\label{prop:cohomologous}
Let $X=\bigsqcup_{\lambda \in \Lambda}G_\lambda$ be an MCQ, 
$R$ a ring and $M$ a left $R$-module.
Let $(f_1,f_2,f_3,f_4;\phi_1,\phi_2)$ and $(g_1,g_2,g_3,g_4;\psi_1,\psi_2)$ 
be 6-tuples of maps satisfying the conditions \eqref{eq:0-i}--\eqref{eq:4-phi}.
If $(f_1,f_2,f_3,f_4;\phi_1,\phi_2) \sim (g_1,g_2,g_3,g_4;\psi_1,\psi_2)$, 
then there exists an MCQ isomorphism $\varphi : \widetilde{X}(f_1,f_2,f_3,f_4;\phi_1,\phi_2) \to \widetilde{X}(g_1,g_2,g_3,g_4;\psi_1,\psi_2)$ 
such that $\pr_X \circ \varphi=\pr_X$ 
for the projection $\pr_X: \bigsqcup_{\lambda \in \Lambda}(G_\lambda \times M) \to X$ sending $(x,u)$ to $x$.
\end{proposition}

\begin{proof}
Assume that $(f_1,f_2,f_3,f_4;\phi_1,\phi_2) \sim_{h,\eta} (g_1,g_2,g_3,g_4;\psi_1,\psi_2)$ 
for some maps $h:X \to R^\times$ and $\eta : X \to M$.
Let $\varphi$ be the map from $\widetilde{X}(f_1,f_2,f_3,f_4;\phi_1,\phi_2)$ to $\widetilde{X}(g_1,g_2,g_3,g_4;\psi_1,\psi_2)$ 
sending $(x,u)$ to $(x,h(x)u+\eta(x))$.
Then 
$\varphi$ is an MCQ isomorphism 
satisfying $\pr_X \circ \varphi=\pr_X$.
\end{proof}

\begin{lemma}\label{simplification of a 6-tuple}
Let $X=\bigsqcup_{\lambda \in \Lambda}G_\lambda$ be an MCQ, 
$R$ a ring and $M$ a left $R$-module.
Let $(f_1,f_2,f_3,f_4;\phi_1,\phi_2)$ be a 6-tuple of maps 
satisfying the conditions \eqref{eq:0-i}--\eqref{eq:4-phi}.
We define the maps $g_1,g_2:X \times X \to R$, 
$g_3,g_4:\bigsqcup_{\lambda \in \Lambda}(G_\lambda \times G_\lambda) \to R$, 
$\psi_1 : X \times X \to M$ 
and $\psi_2:\bigsqcup_{\lambda \in \Lambda}(G_\lambda \times G_\lambda) \to M$ by 
\begin{align*}
&g_1(x,y):=f_1(e_x,y),\\
&g_2(x,y):=f_3(x \tri y, x^{-1}\tri y)f_2(x,y)f_3(e_y,y),\\
&g_3(a,b):=1,\\
&g_4(a,b):=f_1(e_a,a^{-1}),\\
&\psi_1(x,y):=f_3(x \tri y,x^{-1} \tri y)\phi_1(x,y),\\
&\psi_2(a,b):=f_3(ab,b^{-1}a^{-1})\phi_2(a,b).\\
\end{align*}
Then the following hold.
\begin{enumerate}
\item
The 6-tuple $(g_1,g_2,g_3,g_4;\psi_1,\psi_2)$ satisfies 
the conditions \eqref{eq:0-i}--\eqref{eq:4-phi}.
\item
The quadruple $(g_1,g_2;\psi_1,\psi_2)$ is an augmented MCQ Alexander pair.
\end{enumerate}
\end{lemma}

\begin{proof}
\begin{enumerate}
\item
It is sufficient to prove that 
the conditions \eqref{eq:0-phi}, \eqref{eq:1-phi}, \eqref{eq:2-phii}, \eqref{eq:2-phiii}, 
\eqref{eq:3-phi} and \eqref{eq:4-phi} hold 
since 
the remaining conditions follow 
from~\cite[Lemma 4.2]{Murao21} and 
Remark~\ref{replace condition} immediately.

For any $a,b,c \in G_\lambda$, it follows
\begin{align*}
&g_3(ab,c)\psi_2(a,b)+\psi_2(ab,c)\\
&=f_3(ab,b^{-1}a^{-1})\phi_2(a,b)+f_3(abc,c^{-1}b^{-1}a^{-1})\phi_2(ab,c)\\
&=f_3(abc,c^{-1}b^{-1}a^{-1})(f_3(e_\lambda,abc)f_3(ab,b^{-1}a^{-1})\phi_2(a,b)+\phi_2(ab,c))\\
&=f_3(abc,c^{-1}b^{-1}a^{-1})(f_3(ab,c)\phi_2(a,b)+\phi_2(ab,c))\\
&=f_3(abc,c^{-1}b^{-1}a^{-1})(f_4(a,bc)\phi_2(b,c)+\phi_2(a,bc))\\
&=f_3(abc,c^{-1}b^{-1}a^{-1})(f_1(bca,a^{-1})f_3(bc,a)\phi_2(b,c)+\phi_2(a,bc)),
\end{align*}
where the third (resp. fourth) equality comes from \eqref{eq:0-ii} (resp. \eqref{eq:0-phi}), 
and where the fifth equality comes from \eqref{eq:1-i}.
On the other hand, 
it follows 
\begin{align*}
&g_4(a,bc)\psi_2(b,c)+\psi_2(a,bc)\\
&=f_1(e_\lambda,a^{-1})f_3(bc,c^{-1}b^{-1})\phi_2(b,c)+f_3(abc,c^{-1}b^{-1}a^{-1})\phi_2(a,bc)\\
&=f_3(abc,c^{-1}b^{-1}a^{-1})(f_3(e_\lambda,abc)f_1(e_\lambda,a^{-1})f_3(bc,c^{-1}b^{-1})\phi_2(b,c)+\phi_2(a,bc))\\
&=f_3(abc,c^{-1}b^{-1}a^{-1})(f_1(bca,a^{-1})f_3(e_\lambda,bca)f_3(bc,c^{-1}b^{-1})\phi_2(b,c)+\phi_2(a,bc))\\
&=f_3(abc,c^{-1}b^{-1}a^{-1})(f_1(bca,a^{-1})f_3(bc,a)\phi_2(b,c)+\phi_2(a,bc)),
\end{align*}
where the third (resp. fourth) equality comes from \eqref{eq:4-i} (resp. \eqref{eq:0-ii}).
Hence we have 
\begin{align*}
g_3(ab,c)\psi_2(a,b)+\psi_2(ab,c)=g_4(a,bc)\psi_2(b,c)+\psi_2(a,bc),
\end{align*}
which implies that
$(g_1,g_2,g_3,g_4;\psi_1,\psi_2)$ satisfies the condition \eqref{eq:0-phi}.

For any $a,b \in G_\lambda$, it follows
\begin{align*}
&g_4(b,b^{-1}ab)\psi_1(a,b)+\psi_2(b,b^{-1}ab)\\
&=f_1(e_\lambda,b^{-1})f_3(b^{-1}ab,b^{-1}a^{-1}b)\phi_1(a,b)+f_3(ab,b^{-1}a^{-1})\phi_2(b,b^{-1}ab)\\
&=f_3(ab,b^{-1}a^{-1})(f_3(e_\lambda,ab)f_1(e_\lambda,b^{-1})f_3(b^{-1}ab,b^{-1}a^{-1}b)\phi_1(a,b)+\phi_2(b,b^{-1}ab))\\
&=f_3(ab,b^{-1}a^{-1})(f_1(b^{-1}ab^2,b^{-1})f_3(e_\lambda,b^{-1}ab^2)f_3(b^{-1}ab,b^{-1}a^{-1}b)\phi_1(a,b)\\
&\quad+\phi_2(b,b^{-1}ab))\\
&=f_3(ab,b^{-1}a^{-1})(f_1(b^{-1}ab^2,b^{-1})f_3(b^{-1}ab,b)\phi_1(a,b)+\phi_2(b,b^{-1}ab))\\
&=f_3(ab,b^{-1}a^{-1})(f_4(b,b^{-1}ab)\phi_1(a,b)+\phi_2(b,b^{-1}ab))\\
&=f_3(ab,b^{-1}a^{-1})\phi_2(a,b)\\
&=\psi_2(a,b),
\end{align*}
where the third (resp. fourth) equality comes from \eqref{eq:4-i} (resp. \eqref{eq:0-ii}), 
and where the fifth (resp. sixth) equality comes from \eqref{eq:1-i} (resp. \eqref{eq:1-phi}).
Hence $(g_1,g_2,g_3,g_4;\psi_1,\psi_2)$ satisfies the condition \eqref{eq:1-phi}.

For any $x \in X$ and $a,b \in G_\lambda$, it follows
\begin{align*}
g_2(x,e_\lambda)\psi_2(e_\lambda,e_\lambda)
&=f_3(x,x^{-1})f_2(x,e_\lambda)f_3(e_\lambda,e_\lambda)f_3(e_\lambda,e_\lambda)\phi_2(e_\lambda,e_\lambda)\\
&=f_3(x,x^{-1})f_2(x,e_\lambda)\phi_2(e_\lambda,e_\lambda)\\
&=f_3(x,x^{-1})\phi_1(x,e_\lambda)\\
&=\psi_1(x,e_\lambda),
\end{align*}
where the third equality comes from \eqref{eq:2-phii}, 
and 
\begin{align*}
&g_2(x,ab)\psi_2(a,b)+\psi_1(x,ab)\\
&=f_3(x \tri ab, x^{-1} \tri ab)f_2(x,ab)f_3(e_\lambda,ab)f_3(ab,b^{-1}a^{-1})\phi_2(a,b)\\
&\quad+f_3(x \tri ab, x^{-1} \tri ab)\phi_1(x,ab)\\
&=f_3(x \tri ab, x^{-1} \tri ab)(f_2(x,ab)\phi_2(a,b)+\phi_1(x,ab))\\
&=f_3(x \tri ab, x^{-1} \tri ab)(f_1(x \tri a,b)\phi_1(x,a)+\phi_1(x \tri a,b))\\
&=f_1(e_x \tri a,b)f_3(x \tri a,x^{-1} \tri a)\phi_1(x,a)+f_3((x \tri a) \tri b, (x^{-1} \tri a) \tri b)\phi_1(x \tri a,b)\\
&=g_1(x \tri a,b)\psi_1(x,a)+\psi_1(x \tri a,b),
\end{align*}
where the third (resp. fourth) equality comes from \eqref{eq:2-phiii} (resp. \eqref{eq:4-i}).
Hence $(g_1,g_2,g_3,g_4;\psi_1,\psi_2)$ satisfies the conditions \eqref{eq:2-phii} and \eqref{eq:2-phiii}.

For any $x,y,z \in X$, it follows
\begin{align*}
&g_1(x \tri y,z)\psi_1(x,y)+\psi_1(x \tri y,z)\\
&=f_1(e_x \tri y,z)f_3(x \tri y,x^{-1} \tri y)\phi_1(x,y)+f_3((x \tri y) \tri z, (x^{-1} \tri y) \tri z)\phi_1(x \tri y,z)\\
&=f_3((x \tri y) \tri z,(x^{-1} \tri y) \tri z)f_1(x \tri y,z)\phi_1(x,y)\\
&\quad+f_3((x \tri y) \tri z, (x^{-1} \tri y) \tri z)\phi_1(x \tri y,z)\\
&=f_3((x \tri y) \tri z,(x^{-1} \tri y) \tri z)(f_1(x \tri y,z)\phi_1(x,y)+\phi_1(x \tri y,z))\\
&=f_3((x \tri z) \tri (y \tri z),(x^{-1} \tri z) \tri (y \tri z))(f_1(x \tri z, y \tri z)\phi_1(x,z)\\
&\quad+f_2(x \tri z,y \tri z)\phi_1(y,z)+\phi_1(x \tri z,y \tri z))\\
&=f_1(e_x \tri z, y \tri z)f_3(x \tri z,x^{-1} \tri z)\phi_1(x,z)\\
&\quad +f_3((x \tri z) \tri (y \tri z),(x^{-1} \tri z) \tri (y \tri z))f_2(x \tri z,y \tri z)f_3(e_y \tri z,y \tri z)\\
&\qquad \quad f_3(y \tri z,y^{-1} \tri z)\phi_1(y,z)\\
&\quad +f_3((x \tri z) \tri (y \tri z), (x^{-1} \tri z) \tri (y \tri z))\phi_1(x \tri z,y \tri z)\\
&=g_1(x \tri z, y \tri z)\psi_1(x,z)+g_2(x \tri z,y \tri z)\psi_1(y,z)+\psi_1(x \tri z,y \tri z),
\end{align*}
where the second (resp. fourth) equality comes from \eqref{eq:4-i} (resp. \eqref{eq:3-phi}), 
and where the fifth equality comes from \eqref{eq:4-i}.
Hence $(g_1,g_2,g_3,g_4;\psi_1,\psi_2)$ satisfies the condition \eqref{eq:3-phi}.

For any $a,b \in G_\lambda$ and $x \in X$, it follows
\begin{align*}
&g_1(ab,x)\psi_2(a,b)+\psi_1(ab,x)\\
&=f_1(e_\lambda,x)f_3(ab,b^{-1}a^{-1})\phi_2(a,b)+f_3(ab \tri x,b^{-1}a^{-1} \tri x)\phi_1(ab,x)\\
&=f_3(ab \tri x,b^{-1}a^{-1} \tri x)f_1(ab,x)\phi_2(a,b)+f_3(ab \tri x,b^{-1}a^{-1} \tri x)\phi_1(ab,x)\\
&=f_3(ab \tri x,b^{-1}a^{-1} \tri x)(f_1(ab,x)\phi_2(a,b)+\phi_1(ab,x))\\
&=f_3(ab \tri x,b^{-1}a^{-1} \tri x)(f_3(a \tri x,b \tri x)\phi_1(a,x)+f_4(a \tri x,b \tri x)\phi_1(b,x)\\
&\quad+\phi_2(a \tri x,b \tri x))\\
&=f_3(a \tri x,a^{-1} \tri x)\phi_1(a,x)+f_4(a \tri x,a^{-1} \tri x)f_3(b \tri x,b^{-1}a^{-1} \tri x)\phi_1(b,x)\\
&\quad+f_3(ab \tri x,b^{-1}a^{-1} \tri x)\phi_2(a \tri x,b \tri x)\\
&=f_3(a \tri x,a^{-1} \tri x)\phi_1(a,x)\\
&\quad+f_4(a \tri x,a^{-1} \tri x)f_3(e_\lambda \tri x,a^{-1} \tri x)f_3(b \tri x,b^{-1} \tri x)\phi_1(b,x)\\
&\quad+f_3(ab \tri x,b^{-1}a^{-1} \tri x)\phi_2(a \tri x,b \tri x)\\
&=f_3(a \tri x,a^{-1} \tri x)\phi_1(a,x)+f_1(e_\lambda \tri x,a^{-1} \tri x)f_3(b \tri x,b^{-1} \tri x)\phi_1(b,x)\\
&\quad+f_3(ab \tri x,b^{-1}a^{-1} \tri x)\phi_2(a \tri x,b \tri x)\\
&=g_3(a \tri x,b \tri x)\psi_1(a,x)+g_4(a \tri x,b \tri x)\psi_1(b,x)+\psi_2(a \tri x,b \tri x),
\end{align*}
where the second (resp. fourth) equality comes from \eqref{eq:4-i} (resp. \eqref{eq:4-phi}), 
and where the fifth equality comes from \eqref{eq:0-ii} and \eqref{eq:0-iii}, 
and where the sixth (resp. seventh) equality comes from \eqref{eq:0-ii} (resp. \eqref{eq:1-i}).
Hence $(g_1,g_2,g_3,g_4;\psi_1,\psi_2)$ satisfies the condition \eqref{eq:4-phi}.

This completes the proof.

\item
By~\cite[Lemma 4.2]{Murao21}, 
the pair $(g_1,g_2)$ is an MCQ Alexander pair.
We show that 
the pair $(\psi_1,\psi_2)$ is a $(g_1,g_2)$-twisted 2-cocycle.

For any $a,b,c \in G_\lambda$, 
it follows 
\begin{align*}
&\psi_2(a,b)+\psi_2(ab,c)\\
&=f_3(ab,b^{-1}a^{-1})\phi_2(a,b)+f_3(abc,c^{-1}b^{-1}a^{-1})\phi_2(ab,c)\\
&=f_3(abc,c^{-1}b^{-1}a^{-1})f_3(ab,c)\phi_2(a,b)+f_3(abc,c^{-1}b^{-1}a^{-1})\phi_2(ab,c)\\
&=f_3(abc,c^{-1}b^{-1}a^{-1})(f_3(ab,c)\phi_2(a,b)+\phi_2(ab,c))\\
&=f_3(abc,c^{-1}b^{-1}a^{-1})(f_4(a,bc)\phi_2(b,c)+\phi_2(a,bc))\\
&=f_4(a,a^{-1})f_3(bc,c^{-1}b^{-1}a^{-1})\phi_2(b,c)+f_3(abc,c^{-1}b^{-1}a^{-1})\phi_2(a,bc)\\
&=f_4(a,a^{-1})f_3(e_\lambda,a^{-1})f_3(bc,c^{-1}b^{-1})\phi_2(b,c)+f_3(abc,c^{-1}b^{-1}a^{-1})\phi_2(a,bc)\\
&=f_1(e_\lambda,a^{-1})f_3(bc,c^{-1}b^{-1})\phi_2(b,c)+f_3(abc,c^{-1}b^{-1}a^{-1})\phi_2(a,bc)\\
&=g_1(a,a^{-1})\psi_2(b,c)+\psi_2(a,bc),
\end{align*}
where the second (resp. fourth) equality comes from \eqref{eq:0-ii} (resp. \eqref{eq:0-phi}), 
and where the fifth (resp. sixth) equality comes from \eqref{eq:0-iii} (resp. \eqref{eq:0-ii}), 
and where the seventh equality comes from \eqref{eq:1-i}.

For any $a,b \in G_\lambda$, 
it follows 
\begin{align*}
&g_1(b,b^{-1})\psi_1(a,b)+\psi_2(b,b^{-1}ab)\\
&=f_1(e_\lambda,b^{-1})f_3(b^{-1}ab,b^{-1}a^{-1}b)\phi_1(a,b)+f_3(ab,b^{-1}a^{-1})\phi_2(b,b^{-1}ab)\\
&=f_3(a,a^{-1})f_1(b^{-1}ab,b^{-1})\phi_1(a,b)+f_3(ab,b^{-1}a^{-1})\phi_2(b,b^{-1}ab)\\
&=f_3(ab,b^{-1}a^{-1})f_3(a,b)f_1(b^{-1}ab,b^{-1})\phi_1(a,b)+f_3(ab,b^{-1}a^{-1})\phi_2(b,b^{-1}ab)\\
&=f_3(ab,b^{-1}a^{-1})(f_1(b^{-1}ab^2,b^{-1})f_3(b^{-1}ab,b)\phi_1(a,b)+\phi_2(b,b^{-1}ab))\\
&=f_3(ab,b^{-1}a^{-1})(f_4(b,b^{-1}ab)\phi_1(a,b)+\phi_2(b,b^{-1}ab))\\
&=f_3(ab,b^{-1}a^{-1})\phi_2(a,b)\\
&=\psi_2(a,b)
\end{align*}
where the second (resp. third) equality comes from \eqref{eq:4-i} (resp. \eqref{eq:0-ii}), 
and where the fourth (resp. fifth) equality comes from \eqref{eq:4-i} (resp. \eqref{eq:1-i}), 
and where the sixth equality comes from \eqref{eq:1-phi}.

For any $x \in X$ and $a,b \in G_\lambda$, 
it follows 
\begin{align*}
&g_2(x,ab)\psi_2(a,b)+\psi_1(x,ab)\\
&=f_3(x \tri ab,x^{-1} \tri ab)f_2(x,ab)f_3(e_\lambda,ab)f_3(ab,b^{-1}a^{-1})\phi_2(a,b)\\
&\quad+f_3(x \tri ab, x^{-1} \tri ab)\phi_1(x,ab)\\
&=f_3(x \tri ab,x^{-1} \tri ab)(f_2(x,ab)\phi_2(a,b)+\phi_1(x,ab))\\
&=f_3(x \tri ab,x^{-1} \tri ab)(f_1(x \tri a,b)\phi_1(x,a)+\phi_1(x \tri a,b))\\
&=f_1(e_x \tri a,b)f_3(x \tri a,x^{-1} \tri a)\phi_1(x,a)+f_3((x \tri a) \tri b,(x^{-1} \tri a) \tri b)\phi_1(x \tri a,b)\\
&=g_1(x \tri a,b)\psi_1(x,a)+\psi_1(x \tri a,b),
\end{align*}
where the third (resp. fourth) equality comes from \eqref{eq:2-phiii} (resp. \eqref{eq:4-i}).

For any $x,y,z \in X$, 
it follows 
\begin{align*}
&g_1(x \tri y,z)\psi_1(x,y)+\psi_1(x \tri y,z)\\
&=f_1(e_x \tri y,z)f_3(x\tri y,x^{-1}\tri y)\phi_1(x,y)+f_3((x\tri y)\tri z,(x^{-1}\tri y)\tri z)\phi_1(x\tri y,z)\\
&=f_3((x\tri y)\tri z,(x^{-1}\tri y)\tri z)f_1(x\tri y,z)\phi_1(x,y)\\
&\quad+f_3((x\tri y)\tri z,(x^{-1}\tri y)\tri z)\phi_1(x\tri y,z)\\
&=f_3((x\tri y)\tri z,(x^{-1}\tri y)\tri z)(f_1(x\tri y,z)\phi_1(x,y)+\phi_1(x\tri y,z))\\
&=f_3((x\tri y)\tri z,(x^{-1}\tri y)\tri z)(f_1(x\tri z,y\tri z)\phi_1(x,z)+f_2(x\tri z,y\tri z)\phi_1(y,z)\\
&\quad+\phi_1(x\tri z,y\tri z))\\
&=f_1(e_x \tri z,y\tri z)f_3(x\tri z,x^{-1}\tri z)\phi_1(x,z)+f_3((x \tri z)\tri (y\tri z),(x^{-1} \tri z)\tri (y\tri z))\\
&\quad \quad f_2(x \tri z,y\tri z)f_3(e_y \tri z,y \tri z)f_3(y\tri z,y^{-1}\tri z)\phi_1(y,z)\\
&\quad +f_3((x \tri z) \tri (y \tri z),(x^{-1} \tri z) \tri (y \tri z))\phi_1(x \tri z,y\tri z)\\
&=g_1(x \tri z,y\tri z)\psi_1(x,z)+g_2(x \tri z,y\tri z)\psi_1(y,z)+\psi_1(x \tri z,y\tri z),
\end{align*}
where the second (resp. fourth) equality comes from \eqref{eq:4-i} (resp. \eqref{eq:3-phi}), 
and where the fifth equality comes from \eqref{eq:4-i}.

For any $a,b \in G_\lambda$ and $x \in X$, 
it follows 
\begin{align*}
&g_1(ab,x)\psi_2(a,b)+\psi_1(ab,x)\\
&=f_1(e_\lambda,x)f_3(ab,b^{-1}a^{-1})\phi_2(a,b)+f_3(ab \tri x,b^{-1}a^{-1}\tri x)\phi_1(ab,x)\\
&=f_3(ab \tri x,b^{-1}a^{-1}\tri x)f_1(ab,x)\phi_2(a,b)+f_3(ab \tri x,b^{-1}a^{-1}\tri x)\phi_1(ab,x)\\
&=f_3(ab \tri x,b^{-1}a^{-1}\tri x)(f_1(ab,x)\phi_2(a,b)+\phi_1(ab,x))\\
&=f_3(ab \tri x,b^{-1}a^{-1}\tri x)(f_3(a\tri x,b\tri x)\phi_1(a,x)+f_4(a\tri x,b\tri x)\phi_1(b,x)\\
&\quad+\phi_2(a\tri x,b\tri x))\\
&=f_3(a\tri x, a^{-1}\tri x)\phi_1(a,x)+f_4(a \tri x,a^{-1}\tri x)f_3(b\tri x,b^{-1}a^{-1}\tri x)\phi_1(b,x)\\
&\quad+f_3(ab \tri x,b^{-1}a^{-1}\tri x)\phi_2(a\tri x,b\tri x)\\
&=f_3(a\tri x, a^{-1}\tri x)\phi_1(a,x)\\
&\quad+f_4(a \tri x,a^{-1}\tri x)f_3(e_\lambda \tri x,a^{-1}\tri x)f_3(b\tri x,b^{-1}\tri x)\phi_1(b,x)\\
&\quad+f_3(ab \tri x,b^{-1}a^{-1}\tri x)\phi_2(a\tri x,b\tri x)\\
&=f_3(a\tri x, a^{-1}\tri x)\phi_1(a,x)+f_1(e_\lambda \tri x,a^{-1}\tri x)f_3(b\tri x, b^{-1}\tri x)\phi_1(b,x)\\
&\quad+f_3((a\tri x)(b\tri x),(b^{-1}\tri x)(a^{-1}\tri x))\phi_2(a \tri x,b\tri x)\\
&=\psi_1(a,x)+g_1(a \tri x, a^{-1}\tri x)\psi_1(b,x)+\psi_2(a \tri x,b\tri x)
\end{align*}
where the second (resp. fourth) equality comes from \eqref{eq:4-i} (resp. \eqref{eq:4-phi}), 
and where the fifth equality comes from \eqref{eq:0-ii} and \eqref{eq:0-iii}, 
and where the sixth (resp. seventh) equality comes from \eqref{eq:0-ii} (resp. \eqref{eq:1-i}).

Therefore the pair $(\psi_1,\psi_2)$ is a $(g_1,g_2)$-twisted 2-cocycle.
\end{enumerate}
\end{proof}

\begin{theorem}
Let $X=\bigsqcup_{\lambda \in \Lambda}G_\lambda$ be an MCQ, 
$R$ a ring and $M$ a left $R$-module.
For any 6-tuple $(f_1,f_2,f_3,f_4;\phi_1,\phi_2)$ of maps 
satisfying the conditions \eqref{eq:0-i}--\eqref{eq:4-phi}, 
there exists an augmented MCQ Alexander pair $(g_1,g_2; \psi_1,\psi_2)$ 
such that $\widetilde{X}(f_1,f_2,f_3,f_4;\phi_1,\phi_2) \cong \widetilde{X}(g_1,g_2;\psi_1,\psi_2)$.
\end{theorem}

\begin{proof}
Let $(f_1,f_2,f_3,f_4;\phi_1,\phi_2)$ be 
a 6-tuple of maps satisfying the conditions \eqref{eq:0-i}--\eqref{eq:4-phi}, 
and let $g_1,g_2:X \times X \to R$, 
$g_3,g_4:\bigsqcup_{\lambda \in \Lambda}(G_\lambda \times G_\lambda) \to R$, 
$\psi_1 : X \times X \to M$ 
and $\psi_2:\bigsqcup_{\lambda \in \Lambda}(G_\lambda \times G_\lambda) \to M$ 
be the maps defined by 
\begin{align*}
&g_1(x,y):=f_1(e_x,y),\\
&g_2(x,y):=f_3(x \tri y, x^{-1}\tri y)f_2(x,y)f_3(e_y,y),\\
&g_3(a,b):=1,\\
&g_4(a,b):=f_1(e_a,a^{-1}),\\
&\psi_1(x,y):=f_3(x \tri y,x^{-1} \tri y)\phi_1(x,y),\\
&\psi_2(a,b):=f_3(ab,b^{-1}a^{-1})\phi_2(a,b).
\end{align*}
We note that $g_4(a,b)=g_1(a,a^{-1})$.
By Lemma~\ref{simplification of a 6-tuple} (1), 
the 6-tuple $(g_1,g_2,g_3,g_4;\psi_1,\psi_2)$ satisfies the conditions \eqref{eq:0-i}--\eqref{eq:4-phi}.
We define the map $h:X \to R^{\times}$ by $h(x):=f_3(x,x^{-1})$.
Then for any $x,y \in X$, it follows 
\begin{align*}
h(x \tri y)f_1(x,y)
=f_3(x \tri y, x^{-1} \tri y)f_1(x,y)
=f_1(e_x,y)f_3(x,x^{-1})
=g_1(x,y)h(x),
\end{align*}
where the second equality comes from \eqref{eq:4-i}, 
and
\begin{align*}
&h(x \tri y)f_2(x,y)
=f_3(x \tri y, x^{-1} \tri y)f_2(x,y)f_3(e_y,y)f_3(y,y^{-1})
=g_2(x,y)h(y),\\
&h(x \tri y)\phi_1(x,y)
=f_3(x \tri y, x^{-1} \tri y)\phi_1(x,y)
=\psi_1(x,y).
\end{align*}
For any $a,b \in G_\lambda$, it follows 
\begin{align*}
h(ab)f_3(a,b)
=f_3(ab, b^{-1}a^{-1})f_3(a,b)
=f_3(a,a^{-1})
=g_3(a,b)h(a),
\end{align*}
where the second equality comes from \eqref{eq:0-ii}, 
and
\begin{align*}
h(ab)f_4(a,b)
&=f_3(ab,b^{-1}a^{-1})f_4(a,b)f_3(e_\lambda,b)f_3(b,b^{-1})\\
&=f_4(a,a^{-1})f_3(b,b^{-1}a^{-1})f_3(e_\lambda,b)f_3(b,b^{-1})\\
&=f_4(a,a^{-1})f_3(e_\lambda,a^{-1})f_3(b,b^{-1})\\
&=f_1(e_\lambda,a^{-1})f_3(b,b^{-1}) \\
&=g_4(a,b)h(b),
\end{align*}
where the second (resp. third) equality comes from \eqref{eq:0-iii} (resp. \eqref{eq:0-ii}), 
and where the fourth equality comes from \eqref{eq:1-i}, 
and
\begin{align*}
h(ab)\phi_2(a,b)
=f_3(ab,b^{-1}a^{-1})\phi_2(a,b)
=\psi_2(a,b).
\end{align*}
Hence we have $(f_1,f_2,f_3,f_4;\phi_1,\phi_2)\sim_{h,0} (g_1,g_2,g_3,g_4;\psi_1,\psi_2)$, 
where $0$ denotes the zero map.
That is, $\widetilde{X}(f_1,f_2,f_3,f_4;\phi_1,\phi_2) \cong \widetilde{X}(g_1,g_2,g_3,g_4;\psi_1,\psi_2)$ 
by Proposition~\ref{prop:cohomologous}.
By Lemma~\ref{simplification of a 6-tuple} (2), 
$(g_1,g_2;\psi_1,\psi_2)$ is an augmented MCQ Alexander pair.
Therefore we obtain that 
$\widetilde{X}(g_1,g_2,g_3,g_4;\psi_1,\psi_2) = \widetilde{X}(g_1,g_2;\psi_1,\psi_2)$ by the definitions.
This completes the proof.
\end{proof}

\section*{Acknowledgment}
The author would like to thank Atsushi Ishii for his valuable comments.
The author was supported by JSPS KAKENHI Grant Number 18J10105.


\end{document}